\documentclass{amsart}
\usepackage{amssymb}
\usepackage{amsmath}
\usepackage{amsfonts}

\newtheorem{theorem}{Theorem}
\theoremstyle{plain}

\newtheorem{definition}{Definition}

\newtheorem{remark}{Remark}

\numberwithin{equation}{section}
\input{tcilatex}

\begin{document}
\title[Inverse Coefficient Problem for Heat Equation]{Inverse Problem of
Finding the Time-dependent Coefficient of Heat Equation from Integral
Overdetermination Condition Data}
\author{Mansur I. Ismailov$^{\ast )}$ Fatma Kanca$^{\ast \ast )}$}
\address{$^{\ast )}$Department of Mathematics, Gebze Institute of Technology,%
\\
Gebze-Kocaeli 41400, Turkey, Tel/Fax: +902626051641/+902626051365 \\
$^{\ast \ast )}$Department of Mathematics, Kocaeli University, Kocaeli
41380, Turkey }
\email{mismailov@gyte.edu.tr; fzoroglu@kocaeli.edu.tr }
\date{March, 29, 2010}
\dedicatory{}

\begin{abstract}
In this paper we consider the problem of simultaneously determining the
time-dependent thermal diffusivity and the temperature distribution in
one-dimensional heat equation in the case of nonlocal boundary and integral
overdetermination conditions. We establish conditions for the existence and
uniqueness of a classical solution of the problem under considerations. We
present some results on the numerical solution with an example.
\end{abstract}

\maketitle

\section*{Introduction}

\bigskip

\qquad Suppose that one need to determine the temperature distribution $%
u(x,t)$ as well as thermal coefficient $a(t)$ simultaneously satisfy the
equation

\begin{equation}
u_{t}=a(t)u_{xx}+F(x,t),\text{ }0<x<1,\text{ }0<t\leq T,
\end{equation}%
with the initial condition

\begin{equation}
u(x,0)=\varphi (x),\text{ }0\leq x\leq 1,
\end{equation}%
the boundary conditions

\begin{equation}
u(0,t)=u(1,t),\text{ }u_{x}(1,t)=0,\text{ }0\leq t\leq T,
\end{equation}%
and the overdetermination condition

\bigskip 
\begin{equation}
\int\limits_{0}^{1}u(x,t)dx=E(t),\text{ }0\leq t\leq T.
\end{equation}%
The problem of finding the pair $\left\{ a(t),\text{ }u(x,t)\right\} \ $in
(0.1)-(0.4) will be called an inverse problem.

Denote the domain $Q_{T}$\bigskip\ by

\begin{equation*}
Q_{T}=\left\{ \left( x,t\right) :\text{ }0<x<1,\text{ }0<t\leq T\right\} .
\end{equation*}

\begin{definition}
The pair $\left\{ a(t),\text{ }u(x,t)\right\} \ $from the class\ $C\left[ 0,T%
\right] \times C^{2,1}\left( Q_{T}\right) \cap C^{1,0}\left( \overline{Q}%
_{T}\right) $\ for which conditions (0.1)-(0.4) are satisfied and $a(t)>0\ $%
on the interval $\left[ 0,T\right] ,\ $is called the classical solution of
the inverse problem (0.1)-(0.4).
\end{definition}

The parameter identification in a parabolic differential equation from the
data of integral overdetermination condition plays an important role in
engineering and physics. ($\left[ 1,2,3,4,5\right] $)

Various statements of inverse problems on determination of thermal
coefficient in one-dimensional heat equation were studied in [4,5,6]. It is
important to note that in the papers [4,5] the time dependent thermal
coefficient is determined by nonlocal overdetermination condition's data.
Besides, in [1,4] the coefficients of the heat equations are determined in
the case of nonlocal boundary conditions.

In the present work, the existence and uniqueness of the classical solution
of the problem (0.1)-(0.4) is reduced to fixed point principles by applying
Fourier method. The boundary conditions (0.3) admit the expantions by the
system of eigenfunctions and associated functions corresponding to the
spectral problem.

The paper organized as follows:

In Chapter 1, the auxiliary spectral problem which can be obtained by
applying Fourier method to the problem (0.1)-(0.3) is studied. In Chapter 2,
the existence of the solution of inverse problem in $Q_{T},$ and the
uniqueness of the solution of the inverse problem in $Q_{T_{0}}$ $%
(0<T_{0}\leq T)$ are shown. Then in Chapter 3, the continuous dependence
upon the solution of the inverse problem is shown. Finally, in Chapter 4,
the numerical solution for the inverse problem is presented with an example.

\section{The auxiliary spectral problem}

\qquad Consider the spectral problem

\begin{equation*}
X^{\prime \prime }(x)+\lambda X(x)=0,\text{ }0\leq x\leq 1,
\end{equation*}

\begin{equation}
X(0)=X(1),\text{ }X^{\prime }(1)=0.
\end{equation}

This problem is wellknown in [7], as auxiliary spectral problem in solving a
nonlocal boundary value problem for heat equation by Fourier method.

It is clear to show that, the problem (1.1) has eigenvalues

\begin{equation*}
\lambda _{k}=(2\pi k)^{2},k=0,1,2,...
\end{equation*}%
and eigenfunctions

\begin{equation}
\overline{X}_{0}(x)=2,\text{ }\overline{X}_{k}(x)=4\cos 2\pi kx,\text{ }%
k=1,2,...
\end{equation}%
and the system of functions $\overline{X}_{k}(x),$ $k=0,1,2,...$ is not
basis in $L_{2}[0,1]$. Complete the system $\overline{X}_{k}(x),$ $%
k=0,1,2,...$ with the associated functions

\begin{equation}
\overline{\overline{X}}_{k}(x)=4(1-x)\sin 2\pi kx,k=1,2,...
\end{equation}%
of the problem (1.1). Denote the systems of functions (1.2) and (1.3) as
follows:

\begin{equation}
X_{0}(x)=2,\text{ }X_{2k-1}(x)=4\cos 2\pi kx,\text{ }X_{2k}(x)=4(1-x)\sin
2\pi kx,\text{ }k=1,2,...\text{ }.
\end{equation}%
The system of functions $X_{k}(x),$ $k=0,1,2,...$ is basis in $L_{2}[0,1]$.
([8])

The adjoint problem of (1.1) has the form

\begin{equation*}
Y^{\prime \prime }(x)+\lambda Y(x)=0,\text{ }0\leq x\leq 1,
\end{equation*}

\begin{equation}
Y(0)=0,\text{ }Y^{\prime }(0)=Y^{\prime }(1).
\end{equation}%
Analogously to the system (1.4), the system of eigenfunctions and associated
functions of the problem (1.5) is denoted by

\begin{equation}
Y_{0}(x)=x,\text{ }Y_{2k-1}(x)=x\cos 2\pi kx,\text{ }Y_{2k}(x)=\sin 2\pi kx,%
\text{ }k=1,2,....
\end{equation}

\bigskip It is easy to calculate that the systems (1.4) and (1.6) form a
biorthonormal system on interval $\left[ 0,1\right] ,$ i.e.

\begin{equation*}
\left( X_{i},Y_{j}\right) =\int_{0}^{1}X_{i}(x)Y_{j}(x)dx=\delta
_{ij}=\left\{ 
\begin{array}{c}
0,i\neq j \\ 
1,i=j%
\end{array}%
\right. .
\end{equation*}

\section{Existence and Uniqueness of the solution of the inverse problem}

\bigskip \qquad We have the following assumptions on the data of the problem
(0.1)-(0.4).

\begin{enumerate}
\item[(A$_{1}$)] $E(t)\in C^{1}\left[ 0,T\right] ,$ $E^{\prime }(t)<0,$ $%
\forall t\in \left[ 0,T\right] ;$

\item[(A$_{2}$)] $\varphi (x)\in C^{4}\left[ 0,1\right] ;$

\begin{enumerate}
\item[(1)] $\varphi (0)=\varphi (1),$ $\varphi ^{\prime }(1)=0,$ $\varphi
^{\prime \prime }(0)=\varphi ^{\prime \prime }(1),$ $\int\limits_{0}^{1}%
\varphi (x)dx=E(0);$

\item[(2)] $\varphi _{2k}\geq 0,$ $k=1,2,...;$
\end{enumerate}

\item[(A$_{3}$)] $F(x,t)\in C\left( \overline{Q}_{T}\right) ;$ $F(x,t)\in
C^{4}\left[ 0,1\right] $ for arbitrary fixed $t\in \left[ 0,T\right] ;$

\begin{enumerate}
\item[(1)] $F(0,t)=F(1,t),$ $F_{x}(1,t)=0,$ $F_{xx}(0,t)=F_{xx}(1,t);$

\item[(2)] $F_{2k}(t)\geq 0,$ $k=0,1,2,....,\int\limits_{0}^{T}E^{\prime
}(t)dt+\sum\limits_{k=1}^{\infty }\frac{2}{\pi k}\varphi
_{2k}-2\int\limits_{0}^{T}F_{0}(t)dt>0,$
\end{enumerate}
\end{enumerate}

where $\varphi _{k}=\int\limits_{0}^{1}\varphi (x)Y_{k}(x)dx,$ $%
F_{k}(t)=\int\limits_{0}^{1}F(x,t)Y_{k}(x)dx,$ $k=0,1,2,...$ .

\bigskip

\begin{remark}
There are functions $\varphi ,$ $E$ and $F$\ satisfying the assumptions $%
\left( A_{1}\right) -\left( A_{3}\right) .$ For example

$\varphi (x)=(1-x)\sin 2\pi x,$

$E(t)=\frac{1}{2\pi }\exp (-t),$

$F(x,t)=(1-x)\sin 2\pi x\exp (3t).$
\end{remark}

\bigskip

The main result is presented as follows.

\begin{theorem}
Let the assumptions $\left( A_{1}\right) -\left( A_{3}\right) $ be
satisfied. Then the following statements are true:

\begin{enumerate}
\item[(1)] The inverse problem (0.1)-(0.4) has a solution in $Q_{T}$;

\item[(2)] The solution of inverse problem (0.1)-(0.4) is unique in $%
Q_{T_{0}},$ where the number $T_{0}$ $(0<T_{0}<T)$ is determined by the data
of the problem.
\end{enumerate}
\end{theorem}

\begin{proof}
Any solution of the equation (0.1) can be given by

\begin{equation}
u(x,t)=\sum\limits_{k=1}^{\infty }u_{k}(t)X_{k}(x),
\end{equation}%
where the functions $u_{k}(t),$ $k=0,1,2,...$ satisfy the following system
of equations:

\begin{equation*}
u_{0}^{\prime }(t)=F_{0}(t),
\end{equation*}

\begin{equation*}
u_{2k}^{\prime }(t)+\left( 2\pi k\right) ^{2}a(t)u_{2k}(t)=F_{2k}(t),
\end{equation*}

\begin{equation*}
u_{2k-1}^{\prime }(t)+\left( 2\pi k\right) ^{2}a(t)u_{2k-1}(t)+4\pi
ku_{2k}(t)=F_{2k-1}(t),\text{ }k=1,2,...\text{ }.
\end{equation*}

Substuting the solution of this system of equations and initial condition
(0.2) in (2.1), we obtain the solution of the problem (0.1)-(0.3) in the
following form

\begin{eqnarray}
u(x,t) &=&\left[ \varphi _{0}+\int\limits_{0}^{t}F_{0}(\tau )d\tau \right]
X_{0}(x)  \notag \\
&&+\sum\limits_{k=1}^{\infty }\left[ \varphi _{2k}e^{-(2\pi
k)^{2}\int\limits_{0}^{t}a(s)ds}+\int\limits_{0}^{t}F_{2k}(\tau )e^{-(2\pi
k)^{2}\int\limits_{\tau }^{t}a(s)ds}d\tau \right] X_{2k}(x) \\
&&+\sum\limits_{k=1}^{\infty }\left[ \left( \varphi _{2k-1}-4\pi k\varphi
_{2k}t\right) e^{-(2\pi k)^{2}\int\limits_{0}^{t}a(s)ds}\right] X_{2k-1}(x) 
\notag \\
&&+\sum\limits_{k=1}^{\infty }\left[ \int\limits_{0}^{t}\left( F_{2k-1}(\tau
)-4\pi kF_{2k}(\tau )(t-\tau )\right) e^{-(2\pi k)^{2}\int\limits_{\tau
}^{t}a(s)ds}d\tau \right] X_{2k-1}(x).  \notag
\end{eqnarray}

\qquad \qquad \qquad \qquad \qquad \qquad \qquad \qquad \qquad \qquad \qquad
\qquad \qquad \qquad \qquad \qquad \qquad \qquad \qquad \qquad \qquad \qquad
\qquad \qquad \qquad \qquad \qquad \qquad \qquad \qquad \qquad

Under the conditions $(1)$ of $\left( A_{2}\right) $ and $(1)$ of $\left(
A_{3}\right) $ the series (2.2) and $\sum\limits_{k=1}^{\infty }\frac{%
\partial }{\partial x}$ are uniformly convergent in $\overline{Q}_{T}$ since
their majorizing sums are absolutely convergent. Therefore their sums $%
u(x,t) $ and $u_{x}(x,t)$ are continuous in $\overline{Q}_{T}$. In addition,
the series $\sum\limits_{k=1}^{\infty }\frac{\partial }{\partial t}$ and $%
\sum\limits_{k=1}^{\infty }\frac{\partial ^{2}}{\partial x^{2}}$ are
uniformly convergent in $Q_{T}$. Thus, we have $u(x,t)\in C^{2,1}\left(
Q_{T}\right) \cap C^{1,0}\left( \overline{Q}_{T}\right) $. In addition, $%
u_{t}(x,t)$ is continuous in $\overline{Q}_{T}.$ Differentiating (0.4) under
the condition $\left( A_{1}\right) ,$ we obtain

\begin{equation}
\int\limits_{0}^{1}u_{t}(x,t)dx=E^{\prime }(t),\text{ }0\leq t\leq T,
\end{equation}%
and this yields

\begin{equation}
a(t)=P\left[ a(t)\right] ,
\end{equation}%
where%
\begin{equation}
P\left[ a(t)\right] =\frac{2F_{0}(t)+\sum\limits_{k=1}^{\infty }\frac{2}{\pi
k}F_{2k}(t)-E^{\prime }(t)}{\sum\limits_{k=1}^{\infty }8\pi k\left( \varphi
_{2k}e^{-(2\pi
k)^{2}\int\limits_{0}^{t}a(s)ds}+\int\limits_{0}^{t}F_{2k}(\tau )e^{-(2\pi
k)^{2}\int\limits_{\tau }^{t}a(s)ds}d\tau \right) }.
\end{equation}%
Denote 
\begin{equation*}
C_{0}=2\underset{t\in \left[ 0,T\right] }{\min }F_{0}(t)+\underset{t\in %
\left[ 0,T\right] }{\min }\left( \sum\limits_{k=1}^{\infty }\frac{2}{\pi k}%
F_{2k}(t)\right) -\underset{t\in \left[ 0,T\right] }{\max }E^{\prime }(t),
\end{equation*}

\begin{equation*}
C_{1}=2\underset{t\in \left[ 0,T\right] }{\max }F_{0}(t)+\underset{t\in %
\left[ 0,T\right] }{\max }\left( \sum\limits_{k=1}^{\infty }\frac{2}{\pi k}%
F_{2k}(t)\right) -\underset{t\in \left[ 0,T\right] }{\min }E^{\prime }(t),
\end{equation*}

\begin{equation*}
C_{2}=\int\limits_{0}^{T}E^{\prime }(t)dt+\sum\limits_{k=1}^{\infty }\frac{2%
}{\pi k}\varphi _{2k}-2\int\limits_{0}^{T}F_{0}(t)dt,
\end{equation*}

\begin{equation*}
C_{3}=\sum\limits_{k=1}^{\infty }8\pi k\left( \varphi
_{2k}+\int\limits_{0}^{T}F_{2k}(\tau )d\tau \right) .
\end{equation*}
It is easy to verify that $C_{k}>0,$ $k=1,2,3,4$ and $C_{2}\leq C_{3},$ such
that

\begin{equation*}
\int\limits_{0}^{T}E^{\prime }(t)dt+\sum\limits_{k=1}^{\infty }\frac{2}{\pi k%
}\varphi _{2k}-2\int\limits_{0}^{T}F_{0}(t)dt\leq \sum\limits_{k=1}^{\infty }%
\frac{2}{\pi k}\left( \varphi _{2k}e^{-(2\pi
k)^{2}\int\limits_{0}^{T}a(s)ds}+\int\limits_{0}^{T}F_{2k}(\tau )e^{-(2\pi
k)^{2}\int\limits_{\tau }^{T}a(s)ds}d\tau \right) .
\end{equation*}
Using the representation (2.4), the following estimate is true: 
\begin{equation*}
0<\frac{C_{0}}{C_{3}}\leq a(t)\leq \frac{C_{1}}{C_{2}}.
\end{equation*}%
Introduce the set $M$ as: 
\begin{equation*}
M=\left\{ a(t)\in C\left[ 0,T\right] :\frac{C_{0}}{C_{3}}\leq a(t)\leq \frac{%
C_{1}}{C_{2}}\right\} .
\end{equation*}%
It is easy to see that 
\begin{equation*}
P:M\rightarrow M.
\end{equation*}%
Show that the operator P is compact. Let $M_{1}\subset M$ be an arbitrary
bounded set. Since $P(M_{1})\subset M$, then $P(M_{1})$ is uniformly
bounded. Then, we have for $a(t)\in M_{1}$ and $t_{1},$ $t_{2}\in \left[ 0,T%
\right] ,$

\begin{equation}
\left\vert P\left[ a(t_{1})\right] -P\left[ a(t_{2})\right] \right\vert \leq 
\frac{\left\vert K(t_{1})-K(t_{2})\right\vert }{N(t_{2})}+\frac{\left\vert
K(t_{1})\left( N(t_{1})-N(t_{2})\right) \right\vert }{N(t_{1})N(t_{2})},
\end{equation}%
where 
\begin{equation*}
K(t)=2F_{0}(t)+\sum\limits_{k=1}^{\infty }\frac{2}{\pi k}F_{2k}(t)-E^{\prime
}(t),
\end{equation*}

\begin{equation*}
N(t)=\sum\limits_{k=1}^{\infty }8\pi k\left( \varphi _{2k}e^{-(2\pi
k)^{2}\int\limits_{0}^{t}a(s)ds}+\int\limits_{0}^{t}F_{2k}(\tau )e^{-(2\pi
k)^{2}\int\limits_{\tau }^{t}a(s)ds}d\tau \right) .
\end{equation*}%
Using the estimate 
\begin{equation*}
\left\vert e^{-(2\pi k)^{2}\int\limits_{0}^{t_{1}}a(s)ds}-e^{-(2\pi
k)^{2}\int\limits_{0}^{t_{2}}a(s)ds}\right\vert \leq (2\pi k)^{2}\left\vert
t_{1}-t_{2}\right\vert \underset{\left[ 0,T\right] }{\max }a(t),
\end{equation*}%
we obtain 
\begin{equation}
\left\vert N(t_{1})-N(t_{2})\right\vert \leq \left[ \left(
C_{4}+C_{5}\right) \frac{C_{1}}{C_{2}}+C_{6}\right] \left\vert
t_{1}-t_{2}\right\vert ,
\end{equation}%
where

\begin{equation*}
C_{4}=\sum\limits_{k=1}^{\infty }4\left( 2\pi k\right) ^{3}\varphi _{2k},%
\text{ }C_{5}=\int\limits_{0}^{T}\sum\limits_{k=1}^{\infty }4\left( 2\pi
k\right) ^{3}F_{2k}(\tau )d\tau ,\text{ }C_{6}=\underset{t\in \left[ 0,T%
\right] }{\max }\left( \sum\limits_{k=1}^{\infty }8\pi kF_{2k}(t)\right) .
\end{equation*}%
To this end, take an arbitrary $\varepsilon >0.$

Since $K(t)$ is continuous in $\left[ 0,T\right] ,$ then $\exists \delta
_{1}=\delta _{1}(\varepsilon ),$ $\forall t_{1},$ $t_{2}\in \left[ 0,T\right]
$ $\left( \left\vert t_{1}-t_{2}\right\vert <\delta _{1}\right) :$

\begin{equation}
\left\vert K(t_{1})-K(t_{2})\right\vert <\frac{C_{2}\varepsilon }{2}.
\end{equation}
Let 
\begin{equation*}
\delta =\min \left\{ \delta _{1}(\varepsilon ),\text{ }\frac{C_{2}^{3}}{%
2\left( \left( C_{4}+C_{5}\right) C_{1}+C_{2}C_{6}\right) C_{1}}\varepsilon
\right\} .
\end{equation*}%
From (2.7) for $\left\vert t_{1}-t_{2}\right\vert <\delta ,$ we obtain

\begin{equation}
\left\vert N(t_{1})-N(t_{2})\right\vert \leq \frac{C_{2}^{2}}{2C_{1}}%
\varepsilon .
\end{equation}%
Substituting (2.8) and (2.9) in (2.6) we get

\begin{equation*}
\left\vert P\left[ a(t_{1})\right] -P\left[ a(t_{2})\right] \right\vert
<\epsilon .
\end{equation*}

So, the set $P(M_{1})$ is equicontinuous. Then $P(M_{1})$ is a compact set
and the operator $P$ is compact and maps the set $M$ onto itself. Employing
the Schauder's Fixed Point Theorem, we have a solution $a(t)\in C\left[ 0,T%
\right] $ of the equation (2.4).

Now let us show that there exists $Q_{T_{0}}$ $(0<T_{0}\leq T)$ and the
solution $(a,u)$ of the problem (0.1)-(0.4) is unique in $Q_{T_{0}}.$
Suppose that $\left( b,v\right) $ is also a solution pair of the problem
(0.1)-(0.4). Then from the representation (2.2) and (2.4) of the solution,
we have

\begin{equation}
u(x,t)-v(x,t)=\sum\limits_{k=1}^{\infty }\varphi _{2k}\left( e^{-(2\pi
k)^{2}\int\limits_{0}^{t}a(s)ds}-e^{-(2\pi
k)^{2}\int\limits_{0}^{t}b(s)ds}\right) X_{2k}(x)
\end{equation}

\begin{equation*}
+\sum\limits_{k=1}^{\infty }\left( \int\limits_{0}^{t}F_{2k}(\tau )\left(
e^{-(2\pi k)^{2}\int\limits_{\tau }^{t}a(s)ds}-e^{-(2\pi
k)^{2}\int\limits_{\tau }^{t}b(s)ds}\right) d\tau \right) X_{2k}(x)
\end{equation*}

\begin{equation*}
+\sum\limits_{k=1}^{\infty }\left( \varphi _{2k-1}-4\pi k\varphi
_{2k}t\right) \left( e^{-(2\pi k)^{2}\int\limits_{0}^{t}a(s)ds}-e^{-(2\pi
k)^{2}\int\limits_{0}^{t}b(s)ds}\right) X_{2k-1}(x)
\end{equation*}

\begin{equation*}
+\sum\limits_{k=1}^{\infty }\left( \int\limits_{0}^{t}\left( F_{2k-1}(\tau
)-4\pi kF_{2k}(\tau )(t-\tau )\right) \left( e^{-(2\pi
k)^{2}\int\limits_{\tau }^{t}a(s)ds}-e^{-(2\pi k)^{2}\int\limits_{\tau
}^{t}b(s)ds}\right) d\tau \right) X_{2k-1}(x),
\end{equation*}

\begin{equation}
a(t)-b(t)=P\left[ a(t)\right] -P\left[ b(t)\right] ,
\end{equation}
where

\begin{eqnarray*}
P\left[ a(t)\right] -P\left[ b(t)\right] &=&\frac{2F_{0}(t)+\sum%
\limits_{k=1}^{\infty }\frac{2}{\pi k}F_{2k}(t)-E^{\prime }(t)}{%
\sum\limits_{k=1}^{\infty }8\pi k\left( \varphi _{2k}e^{-(2\pi
k)^{2}\int\limits_{0}^{t}a(s)ds}+\int\limits_{0}^{t}F_{2k}(\tau )e^{-(2\pi
k)^{2}\int\limits_{\tau }^{t}a(s)ds}d\tau \right) }- \\
&&\frac{2F_{0}(t)+\sum\limits_{k=1}^{\infty }\frac{2}{\pi k}%
F_{2k}(t)-E^{\prime }(t)}{\sum\limits_{k=1}^{\infty }8\pi k\left( \varphi
_{2k}e^{-(2\pi
k)^{2}\int\limits_{0}^{t}b(s)ds}+\int\limits_{0}^{t}F_{2k}(\tau )e^{-(2\pi
k)^{2}\int\limits_{\tau }^{t}b(s)ds}d\tau \right) }.
\end{eqnarray*}
The following estimate is true:

\begin{eqnarray*}
\left\vert P\left[ a(t)\right] -P\left[ b(t)\right] \right\vert &\leq &\frac{%
\left( 2F_{0}(t)+\sum\limits_{k=1}^{\infty }\frac{2}{\pi k}%
F_{2k}(t)+\left\vert E^{\prime }(t)\right\vert \right) }{C_{2}^{2}}. \\
&&\left( 
\begin{array}{c}
\sum\limits_{k=1}^{\infty }8\pi k\varphi _{2k}\left\vert e^{-(2\pi
k)^{2}\int\limits_{0}^{t}a(s)ds}-e^{-(2\pi
k)^{2}\int\limits_{0}^{t}b(s)ds}\right\vert \\ 
+\sum\limits_{k=1}^{\infty }8\pi k\int\limits_{0}^{T}F_{2k}(\tau )\left\vert
e^{-(2\pi k)^{2}\int\limits_{\tau }^{t}a(s)ds}-e^{-(2\pi
k)^{2}\int\limits_{\tau }^{t}b(s)ds}\right\vert d\tau%
\end{array}%
\right) .
\end{eqnarray*}%
Using the estimates

\begin{eqnarray*}
\left\vert e^{-(2\pi k)^{2}\int\limits_{0}^{t}a(s)ds}-e^{-(2\pi
k)^{2}\int\limits_{0}^{t}b(s)ds}\right\vert &\leq &(2\pi k)^{2}T\underset{%
0\leq t\leq T}{\max }\left\vert a(t)-b(t)\right\vert , \\
\left\vert e^{-(2\pi k)^{2}\int\limits_{\tau }^{t}a(s)ds}-e^{-(2\pi
k)^{2}\int\limits_{\tau }^{t}b(s)ds}\right\vert &\leq &(2\pi k)^{2}T\underset%
{0\leq t\leq T}{\max }\left\vert a(t)-b(t)\right\vert ,
\end{eqnarray*}%
we obtain

\begin{equation*}
\underset{0\leq t\leq T}{\max }\left\vert P\left[ a(t)\right] -P\left[ b(t)%
\right] \right\vert \leq \alpha \underset{0\leq t\leq T}{\max }\left\vert
a(t)-b(t)\right\vert .
\end{equation*}%
Let $\alpha \in (0,1)$ be arbitrary fixed number. Fix a number $T_{0},$ $%
0<T_{0}\leq T,$ such that

\begin{equation*}
\frac{C_{1}(C_{4}+C_{5})}{C_{2}^{2}}T_{0}\leq \alpha .
\end{equation*}%
Then from the equality (2.11) we obtain

\begin{equation*}
\left\Vert a-b\right\Vert _{C\left[ 0,T_{0}\right] }\leq \alpha \left\Vert
a-b\right\Vert _{C\left[ 0,T_{0}\right] },
\end{equation*}%
which implies that $a=b.$ By substituting $a=b$ in (2.10), we have $u=v$.
\end{proof}

Theorem has been proved.

\section{Continuous Dependence of $(a,u)$ upon the data}

\bigskip

\begin{theorem}
Under assumption $\left( A_{1}\right) -\left( A_{3}\right) ,$ the solution $%
(a,u)$ of the problem (0.1)-(0.3) depends continuously upon the data for
small T .
\end{theorem}

\begin{proof}
Let $\ \Phi =\left\{ \varphi ,\text{ }F,\text{ }E\right\} $ and $\overline{%
\Phi }=\left\{ \overline{\varphi },\text{ }\overline{F,}\text{ }\overline{E}%
\right\} $ be two sets of the data, which satisfy the conditions $\left(
A_{1}\right) -\left( A_{3}\right) .$\ Let us denote $\left\Vert \Phi
\right\Vert =(\left\Vert \varphi \right\Vert _{C^{4}\left[ 0,1\right]
}+\left\Vert F\right\Vert _{C^{4,0}(\overline{Q}_{T})}+\left\Vert
E\right\Vert _{C^{1}\left[ 0,T\right] }).$ Suppose that there exist positive
constants $M_{i},$ $i=1,$ $2,$ $3$ such that

\begin{equation*}
\left\Vert \varphi \right\Vert _{C^{4}\left[ 0,1\right] }\leq M_{1},\text{ }%
\left\Vert F\right\Vert _{C^{4,0}(\overline{Q_{T}})}\leq M_{2},\text{ }%
\left\Vert E\right\Vert _{C^{1}\left[ 0,T\right] }\leq M_{3},
\end{equation*}

\begin{equation*}
\left\Vert \overline{\varphi }\right\Vert _{C^{4}\left[ 0,1\right] }\leq
M_{1},\text{ }\left\Vert \overline{F}\right\Vert _{C^{4,0}(\overline{Q_{T}}%
)}\leq M_{2},\text{ }\left\Vert \overline{E}\right\Vert _{C^{1}\left[ 0,T%
\right] }\leq M_{3}.
\end{equation*}

Let $\left( a,u\right) $ and $\left( \overline{a},\overline{u}\right) $ be
solutions of the inverse problem (0.1)-(0.4) corresponding to the data $\Phi 
$ and $\overline{\Phi },$ respectively. According to (2.4),

\begin{equation*}
a(t)=\frac{2F_{0}(t)+\sum\limits_{k=1}^{\infty }\frac{2}{\pi k}%
F_{2k}(t)-E^{\prime }(t)}{\sum\limits_{k=1}^{\infty }8\pi k\left( \varphi
_{2k}e^{-(2\pi
k)^{2}\int\limits_{0}^{t}a(s)ds}+\int\limits_{0}^{t}F_{2k}(\tau )e^{-(2\pi
k)^{2}\int\limits_{\tau }^{t}a(s)ds}d\tau \right) },
\end{equation*}

\begin{equation*}
\overline{a}(t)=\frac{2\overline{F}_{0}(t)+\sum\limits_{k=1}^{\infty }\frac{2%
}{\pi k}\overline{F}_{2k}(t)-\overline{E}^{\prime }(t)}{\sum\limits_{k=1}^{%
\infty }8\pi k\left( \overline{\varphi }_{2k}e^{-(2\pi
k)^{2}\int\limits_{0}^{t}\overline{a}(s)ds}+\int\limits_{0}^{t}\overline{F}%
_{2k}(\tau )e^{-(2\pi k)^{2}\int\limits_{\tau }^{t}\overline{a}(s)ds}d\tau
\right) }.
\end{equation*}%
First let us estimate the difference $a-\overline{a}.$ It is easy to compute
that

\begin{eqnarray*}
&&\left\vert F_{0}(t)\sum\limits_{k=1}^{\infty }8\pi k\overline{\varphi }%
_{2k}e^{-(2\pi k)^{2}\int\limits_{0}^{t}\overline{a}(s)ds}-\overline{F}%
_{0}(t)\sum\limits_{k=1}^{\infty }8\pi k\varphi _{2k}e^{-(2\pi
k)^{2}\int\limits_{0}^{t}a(s)ds}\right\vert \\
&\leq &M_{4}\left\Vert a-\overline{a}\right\Vert _{_{C\left[ 0,T\right]
}}+M_{5}\left\Vert \varphi -\overline{\varphi }\right\Vert _{_{C^{4}\left[
0,1\right] }}+M_{6}\left\Vert F-\overline{F}\right\Vert _{_{C^{4,0}\left( 
\overline{Q}_{T}\right) }},
\end{eqnarray*}

\begin{eqnarray*}
&&\left\vert F_{0}(t)\sum\limits_{k=1}^{\infty }8\pi k\int\limits_{0}^{t}%
\overline{F}_{2k}(\tau )e^{-(2\pi k)^{2}\int\limits_{\tau }^{t}\overline{a}%
(s)ds}d\tau -\overline{F}_{0}(t)\sum\limits_{k=1}^{\infty }8\pi
k\int\limits_{0}^{t}F_{2k}(\tau )e^{-(2\pi k)^{2}\int\limits_{\tau
}^{t}a(s)ds}d\tau \right\vert \\
&\leq &M_{7}\left\Vert a-\overline{a}\right\Vert _{_{C\left[ 0,T\right]
}}+2TM_{5}\left\Vert F-\overline{F}\right\Vert _{_{C^{4,0}\left( \overline{Q}%
_{T}\right) }},
\end{eqnarray*}

\begin{eqnarray*}
&&\left\vert \sum\limits_{k=1}^{\infty }\frac{2}{\pi k}F_{2k}(t)\sum%
\limits_{k=1}^{\infty }8\pi k\overline{\varphi }_{2k}e^{-(2\pi
k)^{2}\int\limits_{0}^{t}\overline{a}(s)ds}-\sum\limits_{k=1}^{\infty }\frac{%
2}{\pi k}\overline{F}_{2k}(t)\sum\limits_{k=1}^{\infty }8\pi k\varphi
_{2k}e^{-(2\pi k)^{2}\int\limits_{0}^{t}a(s)ds}\right\vert \\
&\leq &\frac{M_{4}}{6}\left\Vert a-\overline{a}\right\Vert _{_{C\left[ 0,T%
\right] }}+\frac{M_{5}}{6}\left\Vert \varphi -\overline{\varphi }\right\Vert
_{_{C^{4}\left[ 0,1\right] }}+\frac{M_{6}}{6}\left\Vert F-\overline{F}%
\right\Vert _{_{C^{4,0}\left( \overline{Q}_{T}\right) }},
\end{eqnarray*}

\begin{eqnarray*}
&&%
\begin{array}{c}
\left\vert \sum\limits_{k=1}^{\infty }\frac{2}{\pi k}F_{2k}(t)\sum%
\limits_{k=1}^{\infty }8\pi k\int\limits_{0}^{t}\overline{F}_{2k}(\tau
)e^{-(2\pi k)^{2}\int\limits_{\tau }^{t}\overline{a}(s)ds}d\tau \right. \\ 
\left. -\sum\limits_{k=1}^{\infty }\frac{2}{\pi k}\overline{F}%
_{2k}(t)\sum\limits_{k=1}^{\infty }8\pi k\int\limits_{0}^{t}F_{2k}(\tau
)e^{-(2\pi k)^{2}\int\limits_{\tau }^{t}a(s)ds}d\tau \right\vert%
\end{array}
\\
&\leq &\frac{T^{2}M_{7}}{6}\left\Vert a-\overline{a}\right\Vert _{_{C\left[
0,T\right] }}+\frac{TM_{5}}{3}\left\Vert F-\overline{F}\right\Vert
_{_{C^{4,0}\left( \overline{Q}_{T}\right) }},
\end{eqnarray*}

\begin{eqnarray*}
&&\left\vert E^{\prime }(t)\sum\limits_{k=1}^{\infty }8\pi k\overline{%
\varphi }_{2k}e^{-(2\pi k)^{2}\int\limits_{0}^{t}\overline{a}(s)ds}-%
\overline{E}^{\prime }(t)\sum\limits_{k=1}^{\infty }8\pi k\varphi
_{2k}e^{-(2\pi k)^{2}\int\limits_{0}^{t}a(s)ds}\right\vert \\
&\leq &M_{8}\left\Vert a-\overline{a}\right\Vert _{_{C\left[ 0,T\right]
}}+M_{9}\left\Vert \varphi -\overline{\varphi }\right\Vert _{_{C^{4}\left[
0,1\right] }}+M_{6}\left\Vert E-\overline{E}\right\Vert _{_{C^{1}\left[ 0,T%
\right] }},
\end{eqnarray*}

\begin{eqnarray*}
&&\left\vert E^{\prime }(t)\sum\limits_{k=1}^{\infty }8\pi
k\int\limits_{0}^{t}\overline{F}_{2k}(\tau )e^{-(2\pi
k)^{2}\int\limits_{\tau }^{t}\overline{a}(s)ds}d\tau -\overline{E}^{\prime
}(t)\sum\limits_{k=1}^{\infty }8\pi k\int\limits_{0}^{t}F_{2k}(\tau
)e^{-(2\pi k)^{2}\int\limits_{\tau }^{t}a(s)ds}d\tau \right\vert \\
&\leq &M_{10}\left\Vert a-\overline{a}\right\Vert _{_{C\left[ 0,T\right]
}}+TM_{9}\left\Vert F-\overline{F}\right\Vert _{_{C^{4,0}\left( \overline{Q}%
_{T}\right) }}+TM_{5}\left\Vert E-\overline{E}\right\Vert _{_{C^{1}\left[ 0,T%
\right] }},
\end{eqnarray*}

where $M_{k},$ $k=4,...,12,$ are some constants.

If we consider these estimates in $a-\overline{a},$ we obtain

\begin{equation*}
(1-M_{11})\left\Vert a-\overline{a}\right\Vert _{_{C\left[ 0,T\right] }}\leq
M_{12}\left( \left\Vert E-\overline{E}\right\Vert _{_{C^{1}\left[ 0,T\right]
}}+\left\Vert \varphi -\overline{\varphi }\right\Vert _{_{C^{4}\left[ 0,1%
\right] }}+\left\Vert F-\overline{F}\right\Vert _{_{C^{4,0}\left( \overline{Q%
}_{T}\right) }}\right) .
\end{equation*}%
The inequality $M_{11}<1$ holds for small $T$. Finally, we obtain

\begin{equation*}
\left\Vert a-\overline{a}\right\Vert _{_{C\left[ 0,T\right] }}\leq
M_{13}\left\Vert \Phi -\overline{\Phi }\right\Vert ,\text{ }M_{13}=\frac{%
M_{12}}{(1-M_{11})}.
\end{equation*}
From (2.2), the similar estimate is also obtained for the difference $u-%
\overline{u}$ as

\begin{equation*}
\left\Vert u-\overline{u}\right\Vert _{C\left( \overline{Q}_{T}\right) }\leq
M_{14}\left\Vert \Phi -\overline{\Phi }\right\Vert .
\end{equation*}
\end{proof}

\section{Numerical method and an example}

\qquad We consider an example of numerical solution of the inverse problem
(0.1)-(0.4). We use the finite difference method with a predictor-corrector
type approach, that is suggested in [2]. Apply this method to the problem
(0.1)-(0.4).

We subdivide the intervals $\left[ 0,1\right] $ and $\left[ 0,T\right] $
into M and N subintervals of equal lengths $h=\frac{1}{M}$ and $\tau =\frac{T%
}{N}$ respectively. Then we add a line $x=\left( M+1\right) h$ to generate
the fictitious point needed for dealing with the second boundary condition.
We choose the Crank-Nicolson scheme, which is absolutely stable and has a
second order accuracy in both $h$ and $\tau .$ ([9]) The Crank-Nicolson
scheme for (0.1)-(0.4) is as follows:

\begin{eqnarray}
\frac{1}{\tau }\left( u_{i}^{j+1}-u_{i}^{j}\right) &=&\frac{1}{2}\left(
a^{j+1}+a^{j}\right) \frac{1}{2h^{2}}\left[ \left(
u_{i-1}^{j}-2u_{i}^{j}+u_{i+1}^{j}\right) +\left(
u_{i-1}^{j+1}-2u_{i}^{j+1}+u_{i+1}^{j+1}\right) \right]  \notag \\
&&+\frac{1}{2}\left( F_{i}^{j+1}+F_{i}^{j}\right) ,
\end{eqnarray}

\begin{equation}
u_{i}^{0}=\phi _{i},
\end{equation}

\begin{equation}
u_{0}^{j}=u_{M}^{j},
\end{equation}

\begin{equation}
u_{M-1}^{j}=u_{M+1}^{j},
\end{equation}%
where $1\leq i\leq M$ and $0\leq j\leq N$ are the indices for the spatial
and time steps respectively, $u_{i}^{j}=u(x_{i},t_{j}),$ $\phi _{i}=\varphi
(x_{i}),$ $F_{i}^{j}=F(x_{i},t_{j}),$ $x_{i}=ih,$ $t_{j}=j\tau .$ At the $%
t=0 $ level, adjustment should be made according to the initial condition
and the compatability requirements.

(4.1)-(4.4) problem forms $M\times M$ linear system of equations

\begin{equation}
AU^{j+1}=b,
\end{equation}%
where

$U^{j}=\left( u_{1}^{j},u_{2}^{j},...,u_{M}^{j}\right) ^{T},$ $1\leq j\leq
N, $ $b=(b_{1},b_{2},...,b_{M})^{T},$

\bigskip

$A=%
\begin{bmatrix}
-2\left( 1+R\right) & 1 & 0 & ... &  &  & 0 & 1 \\ 
1 & -2\left( 1+R\right) & 1 & 0 & ... &  &  & 0 \\ 
0 & 1 & -2\left( 1+R\right) & 1 & 0 & ... &  & 0 \\ 
\vdots &  &  & \ddots &  &  &  &  \\ 
&  &  &  &  &  &  &  \\ 
&  &  &  &  &  &  &  \\ 
&  &  &  & 0 & 1 & -2\left( 1+R\right) & 1 \\ 
&  &  &  &  & 0 & 2 & -2\left( 1+R\right)%
\end{bmatrix}%
,$

$R=\frac{2h^{2}}{\tau (a^{j+1}+a^{j})},$ $j=0,1,...,N,$

$b_{1}=2(1-R)u_{1}^{j}-u_{2}^{j}-u_{M}^{j}-R\tau (F_{1}^{j+1}+F_{1}^{j}),$ $%
j=0,1,...,N,$

$b_{M}=-2u_{M-1}^{j}+2(1-R)u_{M}^{j}-R\tau (F_{M}^{j+1}+F_{M}^{j}),$ $%
j=0,1,...,N,$

$b_{i}=-u_{i-1}^{j}+2(1-R)u_{i}^{j}-u_{i+1}^{j}-R\tau
(F_{i}^{j+1}+F_{i}^{j}),$ $i=2,3,...,M-1,$ $j=0,1,...,N.$

Now, let us construct the predicting-correcting mechanism. First,
integrating the equation (0.1) respect to $x$ from $0$ to $1$ and using
(0.3),(0.4), we obtain

\begin{equation}
a(t)=\frac{-E^{\prime }(t)+\int_{0}^{1}F(x,t)dx}{u_{x}(0,t)}.
\end{equation}
The finite difference approximation of (4.6) is

\begin{equation*}
a^{j}=\frac{\left( -\left( Et\right) ^{j}+(Fin)^{j}\right) h}{%
u_{1}^{j}-u_{0}^{j}},
\end{equation*}%
where $\left( Et\right) ^{j}=E^{\prime }(t_{j}),$ $(Fin)^{j}=%
\int_{0}^{1}F(x,t_{j})dx,$ $j=0,1,...,N.$ For $j=0,$

\begin{equation*}
a^{0}=\frac{\left( -\left( Et\right) ^{0}+(Fin)^{0}\right) h}{\phi _{1}-\phi
_{0}},
\end{equation*}%
and the values of $\phi _{i}$ provide us to start our computation. We denote
the values of $a^{j},$ $u_{i}^{j}$ at the s-th iteration step $a^{j(s)},$ $%
u_{i}^{j(s)},$ respectively. In numerical computation, since the time step
is very small, we can take $a^{j+1(0)}=a^{j},$ $u_{i}^{j+1(0)}=u_{i}^{j},$ $%
j=0,1,2,....N,$ $i=1,2,...,M$. At each $(s+1)$-th iteration step we first
determine $a^{j+1(s+1)}$ from the formula

\begin{equation*}
a^{j+1(s+1)}=\frac{\left( -\left( Et\right) ^{j+1}+(Fin)^{j+1}\right) h}{%
u_{1}^{j+1(s)}-u_{0}^{j+1(s)}}.
\end{equation*}%
Then from (4.1)-(4.3) we obtain

\begin{eqnarray}
\frac{1}{\tau }\left( u_{i}^{j+1(s+1)}-u_{i}^{j+1(s)}\right) &=&\frac{1}{%
4h^{2}}\left( a^{j+1(s+1)}+a^{j+1(s)}\right) \left[ \left(
u_{i-1}^{j+1(s)}-2u_{i}^{j+1(s)}+u_{i+1}^{j+1(s)}\right) \right.  \notag \\
&&\left. +\left(
u_{i-1}^{j+1(s+1)}-2u_{i}^{j+1(s+1)}+u_{i+1}^{j+1(s+1)}\right) \right] +%
\frac{1}{2}\left( F_{i}^{j+1}+F_{i}^{j}\right) ,
\end{eqnarray}

\begin{equation}
u_{0}^{j+1(s)}=u_{M}^{j+1(s)},
\end{equation}

\begin{equation}
u_{M-1}^{j+1(s)}=u_{M+1}^{j+1(s)},\text{ }s=0,1,2,...\text{ .}
\end{equation}%
The problem (4.7)-(4.9) can be solved by the Gauss elimination method and $%
u_{i}^{j+1(s+1)}$ is determined.\ If the difference of values on two
iteration reaches the prescribed tolerence, the iteration is stopped and we
accept the corresponding values $a^{j+1(s+1)},$ $%
u_{i}^{j+1(s+1)}(i=1,2,...,M)$ as $a^{j+1},$ $u_{i}^{j+1}$($i=1,2,...,M),$
on the ($j+1)$-th time step, respectively. In virtue of this iteration, we
can move from level $j$ to level $j+1.$

\bigskip

\textbf{Example. }Consider the inverse problem (0.1)-(0.4), with

\begin{eqnarray*}
F(x,t) &=&\left( \frac{1}{\pi }\exp (-t)+4\pi \exp (3t)\right) \cos 2\pi
x+(2\pi )^{2}(1-x)\sin 2\pi x\exp (3t), \\
\varphi (x) &=&(1-x)\sin 2\pi x,\text{ \ \ }E(t)=\frac{1}{2\pi }\exp (-t),%
\text{ \ }T=\frac{1}{4}.
\end{eqnarray*}%
It is easy to check that the exact solution is%
\begin{equation*}
\left\{ a(t),\text{ }u(x,t)\right\} =\left\{ \frac{1}{\left( 2\pi \right)
^{2}}+\exp \left( 4t\right) ,\text{ }(1-x)\sin 2\pi x\exp (-t)\right\} .
\end{equation*}

We use the Crank-Nicolson scheme and the iteration which are explained
above. In result, we obtain Table 1 and Table 2 for exact and approximate
values of $a(t)$ and $u(x,t)$. The step sizes are $h=0.005$ and $\tau =\frac{%
h}{4}.$

\newpage

Table 1. The some values of $a(t)$

\begin{tabular}{llll}
Exact & Approximate & Error & Relative Error \\ 
1.0354 & 1.0261 & 0.0093 & 0.009 \\ 
1.3685 & 1.3276 & 0.0409 & 0.0299 \\ 
1.6658 & 1.6308 & 0.035 & 0.021 \\ 
1.9698 & 1.9487 & 0.0211 & 0.0107 \\ 
2.0798 & 2.0644 & 0.0154 & 0.0074 \\ 
2.2068 & 2.2008 & 0.006 & 0.0027 \\ 
2.3186 & 2.322 & 0.0034 & 0.0014 \\ 
2.4727 & 2.4915 & 0.0188 & 0.0076 \\ 
2.5981 & 2.6304 & 0.0323 & 0.0124 \\ 
2.7301 & 2.7778 & 0.0477 & 0.0175%
\end{tabular}

\bigskip

\newpage

\bigskip

\bigskip

Table 2. The some values of $u(x,t)$ for $T=70$

\begin{tabular}{llll}
Exact & Approximate & Error & Relative Error \\ 
0.3633 & 0.3486 & 0.0147 & 0.0347 \\ 
0.5952 & 0.5933 & 0.0019 & 0.0044 \\ 
0.6506 & 0.6523 & 0.0017 & 0.0001 \\ 
0.4054 & 0.4084 & 0.003 & 0.0017 \\ 
0.3402 & 0.3415 & 0.0013 & 0.0016 \\ 
0.2563 & 0.2554 & 0.0009 & 0.0085 \\ 
-0.2157 & -0.2298 & 0.0141 & 0.0457 \\ 
-0.1430 & -0.1517 & 0.0087 & 0.0066 \\ 
-0.1176 & -0.1252 & 0.0076 & 0.0019 \\ 
-0.0754 & -0.0812 & 0.0059 & 0.0276%
\end{tabular}

\bigskip

\bigskip

\bigskip

\bigskip

\bigskip

\bigskip

\bigskip

\end{document}